\documentclass[12pt]{amsart}

\usepackage{geometry}
\usepackage{amssymb,amsmath,amsfonts,amsthm,mathrsfs,graphicx,microtype}
\usepackage[sans]{dsfont}
\usepackage[dvips,cmtip,all]{xy}
\usepackage{setspace}
\usepackage{hyperref}

\title{The no-three-in-line problem on a torus}

\author{Jim Fowler}
\address[J.~Fowler]{The Ohio State University, Columbus, 43210}
\email[J.~Fowler]{fowler@math.osu.edu}

\author{Andrew Groot}
\address[A.~Groot]{The Ohio State University, Columbus, 43210}
\email[A.~Groot]{groot.2@osu.edu}

\author{Deven Pandya}
\address[D.~Pandya]{The Ohio State University, Columbus, 43210}
\email[D.~Pandya]{pandya.33@osu.edu}

\author{Bart Snapp}
\address[B.~Snapp]{The Ohio State University, Columbus, 43210}
\email[B.~Snapp]{snapp@math.ohio-state.edu}

\theoremstyle{plain}
\newtheorem{thm}{Theorem}[section]
\newtheorem{prop}[thm]{Proposition}
\newtheorem{lem}[thm]{Lemma}

\newtheorem{ques}{Question}
\newtheorem{conj}{Conjecture}

\theoremstyle{definition}
\newtheorem*{dfn}{Definition}

\theoremstyle{remark}

\DeclareMathOperator{\Spec}{Spec}

\renewcommand{\emptyset}{\varnothing}

\newcommand{\N}{\mathbb N}
\newcommand{\Z}{\mathbb Z}

\newcommand{\x}{\mathbf x}

\renewcommand{\l}{\ell}

\renewcommand{\le}{\leqslant}

\newcommand{\hf}{\mathop{\mathrm{HF}}\nolimits}

\begin{document}

\begin{abstract}
  Let $T(\Z_m \times \Z_n)$ denote the maximal number of points that
  can be placed on an $m \times n$ discrete torus with ``no three in a
  line,'' meaning no three in a coset of a cyclic subgroup of $\Z_m
  \times \Z_n$.  By proving upper bounds and providing explicit
  constructions, for distinct primes $p$ and $q$, we show that
  \begin{align*}
    T(\Z_p \times \Z_{p^2}) &= 2p, \\
    T(\Z_p \times \Z_{pq})  &= p+1.
  \end{align*}
  Via Gr\"obner bases, we compute $T(\Z_m \times \Z_n)$ for $2 \leq m
  \leq 7$ and $2 \leq n \leq 19$.
\end{abstract}

\maketitle

\section{Introduction}

In the \textit{no-three-in-line problem} \cite{Dudeney}, one wishes
to place as many points as possible on an $n\times n$ lattice with no
three points on a line.  A predominant conjecture is that $2n$ points
can be placed with no three in a line for all $n\times n$
lattices---note this requires $2$ points for each row (or column), and
hence cannot be improved upon.

As a lower bound, Paul Erd\"os in \cite{kR51} proved that for a
$p\times p$ lattice, $p$ being prime, one can place $p$ points via a
``parabola'' modulo $p$.  This means for $x = 0,\dots,p-1$, no
three of the points $(x, x^2 \bmod{p})$ will be in a line.  Later in
\cite{HJSW75}, this lower bound was improved by considering a
$2p\times 2p$ lattice and placing the points on a ``hyperbola'' modulo
$p$. This construction is somewhat more complex, and for a $2p\times
2p$ lattice, it permits $3p$ points to be placed with no three in a
line. In summary, at this point it is known that $(\frac{3}{2} -
\varepsilon)n$ points can be placed on an $n\times n$ grid.

Instead of attacking this long unsolved problem, we analyze a
variation of it.  Again consider an $n\times n$ lattice, but now
associate opposite edges, so that we may view this as a discrete
$n\times n$ torus (see Figure~\ref{figure:donut}).  We define
\textit{lines} on this discrete torus to be the images of lines in
$\Z\times \Z$ under the covering projection.  We ask the following
question:

\begin{figure}
\label{figure:donut}
\includegraphics{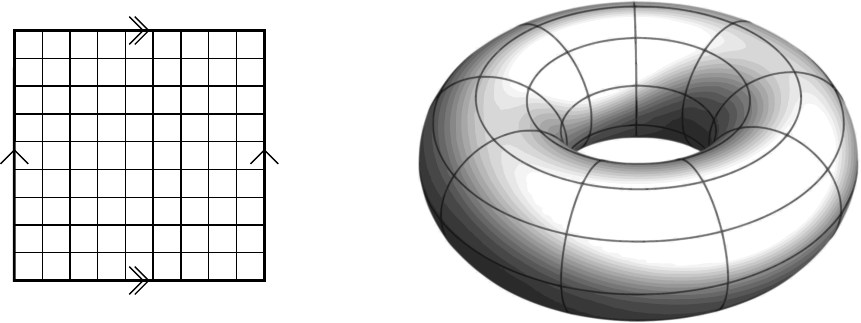}
\begin{caption}
{A $9\times 9$ discrete torus}
\end{caption}
\end{figure}

\begin{ques}
How many points can be placed on an $n\times n$ discrete torus, such
that no three points are in a line?
\end{ques}

In this setting, we reproduce Erd\"os' lower bound for $p \times p$
discrete tori in Theorem~\ref{thm:erdos-analogue}.  However, the
explicit examples in Section \ref{S:ComAlg} show that this lower bound
cannot be improved with the methods of \cite{HJSW75}.  Interestingly,
the size of solutions on tori diverge from those on the lattice almost
immediately. On a $3\times 3$ lattice we may place $6$ points, while
on a corresponding torus, we may only place $4$ points.  Since the
lines can ``wrap around'' the edges, it is harder to place points so
that there are no three in a line.
\begin{figure}
\includegraphics{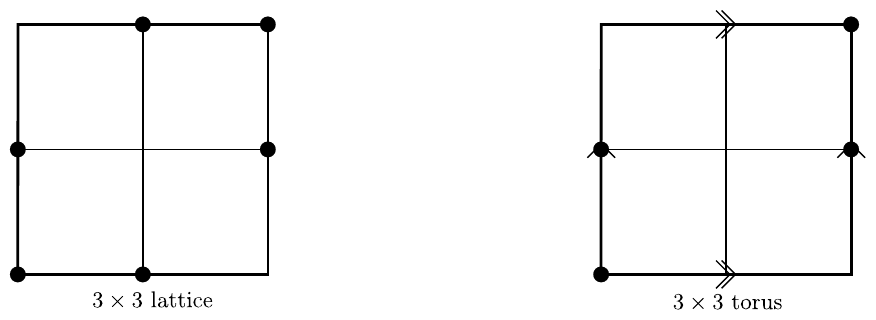}
\begin{caption}
{Maximal solutions on the lattice and the torus}
\end{caption}
\end{figure}

In Section \ref{S:DT}, we give upper and lower bounds (some of which
are given alongside maximal constructions) for the number of points
that can be placed on various $n\times m$ discrete tori with no three
points in a line.  Finally in Section \ref{S:ComAlg}, we will give
some empirical results and give a description of the methods used to
obtain them.

%

\section{Results for discrete tori}\label{S:DT}

To start, note that working with an $n\times n$ torus is essentially a
reformulation of the no-three-in-line problem for the group
$\Z_n\times \Z_n$.

\begin{dfn}
We will say that two points $a = (x_a,y_a)$ and $A
=(x_A,y_A)$ are \textbf{congruent} modulo $n$ if
\[
x_a \equiv x_A \bmod{n} \qquad\text{and}\qquad y_a \equiv y_A \bmod{n} 
\]
and in this case we will simply write $a \equiv A \bmod{n}$.
\end{dfn}

\begin{dfn}\label{D:lines} 
Three distinct points $a$, $b$ and $c$ are \textbf{in a line} on the discrete torus
$\Z_n\times \Z_n$ if and only if there are three points $A$, $B$, and
$C$ in a line in the universal cover $\Z\times \Z$ such that
\[
a \equiv A \bmod{n}, \qquad b \equiv B \bmod{n}, \qquad c \equiv C \bmod{n}.
\]
\end{dfn}

\subsection{Upper bounds}

While it is easy to show that at most $2n$ points can be placed with
no three in a line on a $n\times n$ lattice, this bound is much too
high to be of real use when studying the no-three-in-line problem on
the discrete torus. We arrived at a somewhat general question that
sheds some light on this:

\begin{ques}
Given a group, how many elements of it can be chosen so that no three
are in a coset of a (maximal) cyclic subgroup?
\end{ques}

Essentially, lines on a discrete torus correspond to cosets of cyclic
subgroups of $\Z_n\times \Z_n$. Since we are interested in looking at
whole lines, we can restrict ourselves to looking at cyclic subgroups
that are maximal with respect to set-inclusion.

\begin{dfn} 
Given a group $G$, let $T(G)$ denote the number of elements of $G$
that can be chosen so that no three are in a coset of a cyclic
subgroup of $G$.
\end{dfn}

Note, for any cyclic group $Z$, $T(Z) = 2$, hence when $m$ and $n$ are
relatively prime, $T(\Z_m\times\Z_n) = 2$.

\begin{prop}
For any positive integer $n$, $T(\Z_2 \times \Z_{2n}) = 4$.
\end{prop}

\begin{proof}
Consider the following arrangement of $4$ points:
\[
\includegraphics{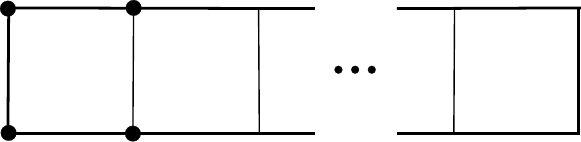}
\]
By inspection we can see that no three elements are in a line on
this torus and hence $T(\Z_2 \times \Z_{2n}) = 4$.
\end{proof}

\begin{prop}
For any positive integers $m$ and $n$, $T(\Z_m\times \Z_m) \le T(\Z_m
\times \Z_{mn})$.
\end{prop}
\begin{proof}
This follows as an $m\times mn$ torus is a cover for an $m\times m$
torus. Hence, the lines that pass through any three points of the
$m\times m$ torus are precisely those that pass through an $m\times m$
section of an $m\times mn$ torus.

\end{proof}

\begin{thm}
For any prime integer $p$, $T(\Z_p \times \Z_p)\le p+1$.
\end{thm}

\begin{proof}
Consider the lines on the $p\times p$ torus. These lines
correspond to cosets of maximal cyclic subgroups of $\Z_p\times
\Z_p$. We claim that there are exactly $p+1$ maximal cyclic subgroups
of $\Z_p\times \Z_p$. To see this, consider the following list of
subgroups:
\[
\{ \langle(0,1)\rangle, \langle(1,0)\rangle, \langle(1,1)\rangle, \langle(1,2)\rangle, \dots, \langle(1,p-1)\rangle\}
\]
We have listed $p+1$ maximal cyclic subgroups of $\Z_p\times
\Z_p$. Moreover, we claim that there are no others. Consider $(a,b)\in
\Z_p\times \Z_p$. If $a$ is zero, then $\langle(a,b)\rangle \leq
\langle(0,1)\rangle$. If $a\ne 0$, then consider the $a$th multiple of
each of the generators above. Since $\Z_p$ is a field, the $a$th
multiple of one of those generators is equal to $(a,b)$, forcing
$\langle(a,b)\rangle$ to be a subset of the maximal cyclic subgroup
generated by that generator.

Since every line is a coset of a maximal cyclic subgroup of
$\Z_p\times \Z_p$, and the cosets of a subgroup partition the group,
every point is in exactly one coset of each maximal cyclic subgroup
listed above. Hence every point on the $p\times p$ torus is contained
in exactly $p+1$ lines.

If we attempt to place points on the $p\times p$ torus such that no
three are in a line, the first point must be on $p+1$ lines, the
second on $p$ lines new lines, the third on $p-1$ new lines, and so on
until the last point which is on just a single new line. No more
points can be placed or we would have three in a line. Hence, at most
$p+1$ points can be placed on a $p\times p$ torus.
\end{proof}

\begin{thm}
For any distinct prime integers $p$ and $q$, $T(\Z_p \times
\Z_{pq})\le p+1$.
\end{thm}

\begin{proof}
Again, the ``lines'' on the $p\times pq$ torus correspond to cosets of
maximal cyclic subgroups of $\Z_p\times \Z_{pq}$. We claim that there
are exactly $p+1$ maximal cyclic subgroups of $\Z_p\times
\Z_{pq}$. Consider the following list of subgroups:
\[
\{ \langle(0,1)\rangle, \langle(1,1)\rangle, \langle(2,1)\rangle, \dots, \langle(p-1,1)\rangle, \langle (1, p)\rangle\}
\]
We claim that these $p+1$ subgroups are all of the maximal cyclic
subgroups of $\Z_p\times \Z_{pq}$. Consider $(a,b)\in \Z_p\times
\Z_{pq}$. If $a$ is zero, then $\langle(a,b)\rangle \leq
\langle(0,1)\rangle$. If $a\ne 0$, and $p\nmid b$, then consider the
$b$th multiple of each of the generators above. Since $\Z_p$ is a
field, the $b$th multiple of one of those generators is equal to
$(a,b)$. If $p|b$ then consider the $ib/p$th multiples of $(1,p)$
where $i\in\{0,\dots, p-1\}$. Since $(p,q)= 1$, we see that
$\langle(a,b)\rangle$ to be a subset of the maximal cyclic subgroup
generated by $(1,p)$. Working as in the proof of the previous theorem
we see that at most $p+1$ points can be placed on a $p\times pq$
torus.
\end{proof}

\subsection{Constructions}

A construction originally given by Erd\"os in \cite{kR51} shows that
if $p$ is prime, we may place $p$ points on a $p\times p$ discrete
torus such that no three are in a line.  To see this, recall the
determinant criterion for checking whether points are in a line:

\begin{lem}\label{L:det}
Three points $(x_1,y_1)$, $(x_2,y_2)$, and $(x_3,y_3)$ are in line if
and only if
\[
\det \begin{bmatrix}
1 & 1 & 1\\
x_1 & x_2 & x_3 \\
y_1 & y_2 & y_3
\end{bmatrix} =0.
\]
\end{lem}

Using this lemma we will adapt the proof given in \cite{AHK74} to
prove the following theorem.

\begin{thm} 
\label{thm:erdos-analogue}
Given a prime $p$ and the discrete torus $\Z_{p}\times \Z_{p}$,
there are $p$ points none of which are three-in-line.
\end{thm}

\begin{proof} 
Consider the set of points:
\[
\{(x, x^2 \bmod p):x=0,\dots,p-1\}
\]
By Lemma~\ref{L:det} we should examine the following determinant
\[
\det \begin{bmatrix}
1 & 1 & 1\\
x + pa & y+pb & z+pc \\
x^2+pi & y^2+pj & z^2+pk
\end{bmatrix} 
\]
which equals
\[
(y-x)(x-z)(y-z) + p(\text{other terms}).
\]
The first term is nonzero and not divisible by $p$ because $x$, $y$,
and $z$ are distinct elements of $\{0,\dots,p-1\}$. Thus the
determinant in question is neither nonzero nor is it divisible by
$p$. Thus we have shown that $p$ points can be placed on the $p\times
p$ discrete torus with no three-in-line.
\end{proof}

The construction above places $p$ points on either the discrete torus
or lattice. However, in neither case is the construction maximal. The
following constructions are all maximal.

\begin{thm}
For any prime integer $p$, $T(\Z_p \times \Z_{p^2}) = 2p$.
\end{thm}
\begin{proof}
The proof uses a construction similar to Erd\"os' construction for the
$p \times p$ lattice. Consider the set of points:
\[
X=\{(x, p x^2 \bmod{p^2}):x=0,\dots,p-1\}
\]
along with
\[
Y=\{(p-x-1, -px^2-1,\bmod{p^2}):x=0,\dots,p-1\}
\]
Here $Y$ is essentially an $180$ degree rotation of the points in
$X$. We claim that together these sets produce $p-1$ points where no
three are in a line. First we must argue that these sets are disjoint.

Seeking a contradiction, suppose $X\cap Y \neq \emptyset$, then plugging the
first entry of a point of $Y$ into the formula for the second
entry of a point of $X$ will equal the second entry of a
point of $Y$. Writing this out:
\begin{align*}
p(p-x-1)^2 &\equiv -px^2 - 1 \bmod{p^2} \\
px^2 +2px + p&\equiv -px^2 - 1 \bmod{p^2} \\
2px^2 + 2px +p +1 &\equiv 0 \bmod{p^2} 
\end{align*}
However, multiplying both sides by $p$, we see this to be impossible.

Now we claim that no three of the $2p$ points of $X\cup Y$ are in a
line. If there were three points in a line, then either all of those
points are from $X$, all are from $Y$, or two are from one set and the
third is from the other set. Since $Y$ is merely a $180$ degree
rotation of the first set of points, we can work as we did before and
examine the following determinant:
\[
\det
\begin{bmatrix}
1           & 1            & 1 \\
x+ap        & y+bp         & z+cp \\
px^2+ ip^2  & py^2+jp^2    & pz^2+kp^2 \\
\end{bmatrix}
\]
On the other hand, if one point is from $Y$ and two points are from
$X$, or vice versa, we examine this determinant.
\[
\det
\begin{bmatrix}
1           & 1            & 1 \\
p-x-1+ap        & y+bp         & z+cp \\
-px^2-1+ ip^2  & py^2+jp^2    & pz^2+kp^2 \\
\end{bmatrix}
\]
By symmetry, these two determinants are sufficient to account for all
cases.  The first determinant above is equal to:
\[
-p(x-y)(x-z)(y-z)+p^2(\text{other terms})
\]
Since $x$, $y$, and $z$ are distinct elements of $\{0\dots ,p-1\}$
this first term must be nonzero and not divisible by $p^2$.  Thus the
determinant is nonzero.  The other determinant equals
\[
(y-z) + p(\text{other terms})
\]
which by the same logic is also nonzero.  Thus we have shown that $2p$
points can be placed on a the $p \times p^2$ torus.
\end{proof}

Next we give construction for placing $p+1$ points on a
$p\times p$ torus.

\begin{thm} 
For any prime $p$,
  \[
  T(\Z_p \times \Z_p) = p+1.
  \]
\end{thm}

\begin{proof} 
  Since $T(\Z_p \times \Z_p) \leq p+1$, constructing an arrangement of
  $p+1$ points will suffice to prove the theorem.  The construction
  relies on counting points on spheres for quadratic forms over finite
  fields, for which we referred to Cassleman's survey \cite{casselman}
  of Minkowski's counting arguments \cite{minkowski}.

  If $p = 2$, any configuration of 3 points works.  For $p > 2$, we
  begin by choosing a quadratic nonresidue $q$.  Regarding $\Z_p
  \times \Z_p$ as the affine plane over the finite field $\Z_p$, the
  variety
  \[
  V := \{ (x,y) \in {\Z_p}^2 : x^2 + q \, y^2 = 1 \}.
  \]
  is an absolutely irreducible degree two hypersurface; if it were
  reducible over the algebraic closure $\overline{\Z_p}$, the
  irreducible components of the projective closure of $\Spec
  \overline{\Z_p}[x,y]/(x^2 + q\, y^2 - 1)$ would intersect by B\'ezout,
  giving a singular point, but the homogeneous polynomial $x^2 + q\, y^2 -
  z^2$ has partial derivatives which simultaneously vanish only at
  $(0,0,0)$, so the projective closure is nonsingular.

  Also by B\'ezout's theorem, any line (a degree one hypersurface)
  intersects $V$ in at most two points.  In other words, $V$ satisfies
  the no-three-in-line condition.  It remains to count the points on
  $V$.  Define the finite field extension $k := \Z_p[t]/(t^2 - q)$
  having $p^2$ elements, and consider the norm map $N : k \to \Z_p$.
  Regarding $k$ as a two-dimensional vector space over $\Z_p$, we may
  identify $V$ with the preimage $N^{-1}(1)$.

  Let $F : k \to k$ be the Frobenius; then
  \[
  N(x) = x \cdot F(x) = x^{1+p},
  \]
  and the units $k^{\times}$ is a cyclic group, so $N$ is surjective.
  Because $N$ is a group homomorphism on nonzero elements, the fiber
  over each nonzero element of $\Z_p$ has the same number of elements,
  so the fiber has size $(p^2 - 1)/(p-1) = p+1$.  This $V$ has $p+1$
  points, as desired.
\end{proof}

\begin{thm} 
For distinct odd primes $p$ and $q$, $T(\Z_p \times \Z_{pq})= p+1$.
\end{thm}

\begin{proof}
The proof uses a similar construction to the one used for the $p
\times p^2$ torus. Consider the set of points:
\[
X=\{(q x^2 \bmod p, p x^4 \bmod{p q}):x=0,\dots,(p-1)/2\}
\]
along with
\[
Y=\{(p-1)/2-q x^2 \bmod p, q (p-1)^2/4 - p x^4 \bmod{p q}):x=0,\dots,(p-1)/2\}
\]
Again, points in $Y$ are essentially an $180$ degree rotation of the
points in $X$. First we must show that $X\cap Y = \emptyset$.  Suppose
that $X\cap Y \neq \emptyset$, then for some values of $x$ and $y$,
\[
q x^2 \equiv \frac{p-1}{2} - q y^2 \pmod{p} \qquad \Rightarrow \qquad x^2 \equiv \frac{p-1}{2q} - y^2 \pmod{p}
\]
and
\[
p x^4 \equiv \frac{q (p-q)^2}{4} - py^4 \pmod{pq}
\]

Combining the equations above:
\[
p\left(\frac{p-1}{2q} - y^2\right)^2 \equiv \frac{q(p-1)^2}{4} - py^4 \pmod{pq}
\]
Multiplying by $q$:
\begin{align*}
0 &\equiv \frac{q^2(p-q)^2}{4} + q \pmod{pq} \\
0 &\equiv q^2 + 4q \pmod{pq}
\end{align*}
which, is impossible.

Again we claim that no three of the $p+1$ points of $X\cup Y$ are in a
line and we examine the following determinant:
\[
\det
\begin{bmatrix}
1           & 1             & 1 \\
q x^2+ap    & q y^2+bp      & q z^2+cp \\
p x^4+ipq   & p y^4+jpq     & p z^4+kpq \\
\end{bmatrix}
\]
On the other hand, if one point is from $Y$ and two points are from
$X$, or vice versa, we examine this determinant.
\[
\det
\begin{bmatrix}
1                       & 1             & 1 \\
(p-1)/2-ap-qx^2         & q y^2+bp      & q z^2+cp \\
q(p-1)^2/4-ipq-px^4   & p y^4+jpq     & q z^4+kpq \\
\end{bmatrix}
\]
By symmetry, these two determinants are sufficient to account for all
cases.  The first determinant above is equal to:
\[
p^2(c-b)x^4+(a-c)y^4+(b-a)z^4-pq(x-y)(x+y)(x-z)(y-z)(x+z)(y+z)
\]
Here we choose to work mod $p^2$, allowing us to ignore the first term 
at the expense of potential roots.  Thus we need only show:
\[
pq(x-y)(x+y)(x-z)(y-z)(x+z)(y+z)
\]
cannot be zero. As above, since $x$, $y$, and $z$ are distinct elements 
of $\{0,\dots,(p-1)/2\}$ none of the differences can be zero. Thus the
determinant is nonzero.  The other determinant equals
\[
-q^2(y-z)(y+z)/4 + p(\text{other terms})
\]
which by the same logic is also nonzero (mod $p$).  Thus we have shown that $p+1$
points can be placed on a the $p \times p q$ torus.
\end{proof}

\section{Commutative algebra}
\label{S:ComAlg}

Before we found the upper-bounds and constructions described above,
our work on this problem was mostly computer-based. However our
approach was somewhat different than what was done in \cite{CHJ76,
  kT78, kT79,aF92,aF98}. Since we did not know have upper bounds for
the number of points that could be placed on an $n\times m$ discrete
torus with no-three-in-line, we could not search for solutions and
stop when a maximal solution was found. To remedy this, we used the
tools of commutative ring theory.  Let $K$ be a field and consider the
polynomial ring:
\[
K[x_{1,1},\dots,x_{n,n}]
\]
By thinking of each indeterminate $x_{i,j}$ as the point $(i,j)$ on
the $n\times n$ lattice or discrete torus, we can use the tools of
commutative algebra to attack these combinatorial problems. While the
use of commutative algebra in combinatorics is not new
\cite{rS96,mK05}, this is the first time that we are aware of that
such methods have been used in connection to the no-three-in-line
problem. In what follows below, $K = \Z_2$ and we will always be
working with a quotient ring
\[
R = K[x_{1,1},\dots,x_{n,n}]/I
\]
where $I$ is an ideal generated by a set of ``undesirable''
points. Specifically, $I$ will contain all products of indeterminates
representing ``three points in a line,'' and squares of every
indeterminate of $K[x_{1,1},\dots,x_{n,n}]$.  As an example, for the
$3\times 3$ lattice,
\begin{align*}
I_\l = ( &x_{1,1}x_{2,1}x_{3,1}, x_{1,1}x_{1,2}x_{1,3}, x_{1,2}x_{2,2}x_{3,2}, x_{2,1}x_{2,2}x_{2,3}, \\
       &x_{1,3}x_{2,3}x_{3,3}, x_{3,1}x_{3,2}x_{3,3}, x_{1,3}x_{2,2}x_{3,1}, x_{1,1}x_{2,2}x_{3,3}, \\
       &x_{1,1}^2, x_{1,2}^2, x_{1,3}^2, x_{2,1}^2, x_{2,2}^2, x_{2,3}^2, x_{3,1}^2, x_{3,2}^2, x_{3,3}^2)
\end{align*}
Looking at the subscripts we see the vertical, horizontal and diagonal
lines on the $3\times 3$ lattice represented as degree three
monomials. Of course, for larger $n$ there are many more lines and
therefore many more such products in the ideal.  Next we see perfect
square monomials, representing the fact that no point can occupy the
same spot twice. On the torus, we have $4$ extra monomials in the
ideal:
\begin{align*}
I = I_\l + ( &x_{1,1}x_{2,3}x_{3,2}, x_{1,2}x_{2,1}x_{3,3}, x_{1,2}x_{2,3}x_{3,1}, x_{1,3}x_{2,1}x_{3,2})
\end{align*}
If one inspects these monomials, we see that they correspond exactly
to lines on the torus that do not exist on the lattice. Hence we see
that monomials of degree $d$ in $R$ will correspond to arrangements of
$d$ points on the discrete torus where no three of those points are in
a line. To see how this setup will allow us to attack this problem, we
need the following well-known definitions; while we restrict ourselves
to the setting of our work, the curious reader may consult \cite{KR05}
for a complete development.

\begin{dfn}
The \textbf{Hilbert function} $\hf_K:\N\to \N$ is defined by
\[
\hf_R(d) := \dim_K(R_d)
\]
where  $R_d$ is the $K$-vector subspace of homogeneous polynomials of degree $d$.
\end{dfn}
In our setting $R = K[\x]/I$, hence a degree $d$ basis of $R$ is a
list of all arrangements of $d$ points on the $n\times n$ lattice or
torus, with no three in a line. Thus $\hf_R(d)$ will correspond to the
number of arrangements of $d$ points on the $n\times n$ discrete torus
with no three in a line.

It is important to notice that since the ideal $I$ in our definition
of $R$ will always contain the squares of each indeterminate of
$K[x_{1,1},\dots,x_{n,n}]$, we see that $\hf_R(d) = 0$ whenever
$d>n^2$.  As such, we can rephrase the no-three-in-line problem as the
following:

\begin{ques}
With $R = K[\x]/I$ as defined above, what is the greatest degree $d$
such that $\hf_R(d) \neq 0$?
\end{ques}

With this re-phrasing in mind, we wish to obtain as much information
as possible regarding the Hilbert function of $R$. Hence we are
interested in the generating function for the Hilbert function, known
as the \textit{Hilbert series} of $R$:

\begin{dfn}
The \textbf{Hilbert series} of a quotient of a polynomial ring $R =
K[\x]/I$, is a power series whose degree $n$ coefficients are exactly
$\hf_R(n)$.
\end{dfn}


A possible advantage to using Hilbert series to study the
no-three-in-line problem, especially over the majority of methods that
are seen elsewhere, is that they give information about placing any
number of points---not just a maximum number of points as the
coefficient of the degree $k$ term is exactly the number of ways that
$k$ points can be placed on an $n \times n$ lattice or torus. Since
our ideal contains the square of every indeterminate, our Hilbert
series will always have finite degree, and hence will be a
polynomial. The degree of this polynomial will always be the size of
the largest possible solution to the no-three-in-line problem on a
lattice or torus.

By encoding this problem in the language of commutative ring theory,
we were able to use the computer algebra system \texttt{Macaulay2},
\cite{M2}, to compute the Hilbert function, Hilbert series, and
relevant bases of our rings.

\subsection{A survey of our findings}

We have been able to reproduce some of the known results on the
no-three-in-line problem for the $n\times n$ lattice via computations
involving the ideals above and their corresponding Hilbert series.  By
looking at the degree of the highest order term and it's coefficient,
we found the highest number of points that can be placed on the
lattice and the number of solutions of that size, respectively.  The
following table lists these data for the first 5 non-trivial cases:
\[
    \begin{array}{c || c | c}
        n & \#\text{ of Points} & \#\text{ of Solutions} \\ \hline\hline
        3 & 6 & 2 \\
        4 & 8 & 11 \\
        5 & 10 & 32 \\
        6 & 12 & 50 \\
        7 & 14 & 132 \\
        \vdots & \vdots & \vdots \\
    \end{array}
\]
Here is the same table for tori:
\[
    \begin{array}{c || c | c}
        n & \#\text{ of Points} & \#\text{ of Solutions} \\ \hline\hline
        3 & 4 & 6 \\
        4 & 6 & 2 \\
        5 & 6 & 40 \\
        6 & 8 & 6 \\
        7 & 8 & 126 \\
        \vdots & \vdots & \vdots \\
    \end{array}
\]
Comparing the two tables, one of the most striking differences is the
surprisingly low number of solutions for even tori while the odd tori
stay fairly close in number to the lattice solutions.  More unexpected
is the size of the solutions, with the tori containing fewer points
every time.  Another interesting detail is the repetition of solutions
sizes.  The progression of torus solution sizes is:
\[
1,4,4,6,6,8,8,8,9,12,12,\dots
\]
However, one anomaly is far more intriguing than the rest.  For the
$14 \times 14$ discrete torus, only 12 points can be placed.  This is
of particular interest as the size of the torus exceeds the size of
its maximal solutions.  For rectangular tori, we have collected the
following data:
\[
    \begin{array}{ c || c | c | c | c | c | c | c | c | c | c | c | c | c | c | c | c | c | c |}
        m\backslash n & 2 & 3 & 4 & 5 & 6 & 7 & 8 & 9 & 10 & 11 & 12 & 13 & 14 & 15 & 16 & 17 & 18 & 19 \\ \hline \hline
        2      & 4 & 2 & 4 & 2 & 4 & 2 & 4 & 2 & 4  & 2  & 4  & 2  & 4  & 2  & 4  & 2  & 4  & 2  \\ \hline
        3      &   & 4 & 2 & 2 & 4 & 2 & 2 & 6 & 2  & 2  & 4  & 2  & 2  & 4  & 2  & 2  & 6  & 2  \\ \hline
        4      &   &   & 6 & 2 & 4 & 2 & 8 & 2 & 4  & 2  & 6  & 2  & 4  & 2  & 8  & 2  & 4  & 2  \\ \hline
        5      &   &   &   & 6 & 2 & 2 & 2 & 2 & 6  & 2  & 2  & 2  & 2  & 6  & 2  & 2  & 2  & 2  \\ \hline
        6      &   &   &   &   & 8 & 2 & 4 & 6 & 4  & 2  & 8  & 2  & 4  & 4  & 4  & 2  & 10 & 2  \\ \hline
        7      &   &   &   &   &   & 8 & 2 & 2 & 2  & 2  & 2  & 2  & 8  & 2  & 2  & 2  & 2  & 2  \\ \hline
    \end{array}
\]
This table shows how many points can be placed on the $m\times n$
torus, with no three in a line. This data was used in formulating the
conjectures that eventually became our maximal constructions above.

\bibliographystyle{alpha}
\bibliography{sources.bib}

\begin{thebibliography}{HJSW75}

\bibitem[AHK74]{AHK74}
Michael~A. Adena, Derek~A. Holton, and Patrick~A. Kelly.
\newblock Some thoughts on the no-three-in-line problem.
\newblock In {\em Combinatorial mathematics ({P}roc. {S}econd {A}ustralian
  {C}onf., {U}niv. {M}elbourne, {M}elbourne, 1973)}, pages 6--17. Lecture Notes
  in Math., Vol. 403. Springer, Berlin, 1974.

\bibitem[Cas]{casselman}
Bill Casselman.
\newblock Quadratic forms over finite fields.
\newblock Available at \url{www.math.ubc.ca/~cass/siegel/FiniteFields.pdf}.

\bibitem[CHJ76]{CHJ76}
D.~Craggs and R.~Hughes-Jones.
\newblock On the no-three-in-line problem.
\newblock {\em J. Combinatorial Theory Ser. A}, 20(3):363--364, 1976.

\bibitem[Dud59]{Dudeney}
Henry~Ernest Dudeney.
\newblock {\em Amusements in mathematics}.
\newblock Dover Publications Inc., New York, 1959.

\bibitem[Fla92]{aF92}
Achim Flammenkamp.
\newblock Progress in the no-three-in-line problem.
\newblock {\em J. Combin. Theory Ser. A}, 60(2):305--311, 1992.

\bibitem[Fla98]{aF98}
Achim Flammenkamp.
\newblock Progress in the no-three-in-line problem. {II}.
\newblock {\em J. Combin. Theory Ser. A}, 81(1):108--113, 1998.

\bibitem[GS]{M2}
Daniel~R. Grayson and Michael~E. Stillman.
\newblock Macaulay2, a software system for research in algebraic geometry.
\newblock Available at \url{http://www.math.uiuc.edu/Macaulay2/}.

\bibitem[HJSW75]{HJSW75}
R.~R. Hall, T.~H. Jackson, A.~Sudbery, and K.~Wild.
\newblock Some advances in the no-three-in-line problem.
\newblock {\em J. Combinatorial Theory Ser. A}, 18:336--341, 1975.

\bibitem[Kat05]{mK05}
Mordechai Katzman.
\newblock Counting monomials.
\newblock {\em J. Algebraic Combin.}, 22(3):331--341, 2005.

\bibitem[Kl{\o}78]{kT78}
Torleiv Kl{\o}ve.
\newblock On the no-three-in-line problem. {II}.
\newblock {\em J. Combinatorial Theory Ser. A}, 24(1):126--127, 1978.

\bibitem[Kl{\o}79]{kT79}
Torleiv Kl{\o}ve.
\newblock On the no-three-in-line problem. {III}.
\newblock {\em J. Combin. Theory Ser. A}, 26(1):82--83, 1979.

\bibitem[KR05]{KR05}
Martin Kreuzer and Lorenzo Robbiano.
\newblock {\em Computational commutative algebra. 2}.
\newblock Springer-Verlag, Berlin, 2005.

\bibitem[MSW11]{minkowski}
H.~Minkowski, Andreas Speiser, and Hermann Weyl.
\newblock {Gesammelte Abhandlungen von {\it Hermann Minkowski}. Unter
  Mitwirkung von {\it Andreas Speiser} und {\it Hermann Weyl} herausgegeben von
  {\it David Hilbert}. Erster Band. Mit einem Bildnis {\it Hermann Minkowskis}
  und 6 Figuren im Text. XXXVI u. 371 S. Zweiter Band. Mit einem Bildnis {\it
  Hermann Minkowskis}, 34 Figuren in Text und einer Doppeltafel. IV u. 466 S.}
\newblock 1911.

\bibitem[Rot51]{kR51}
K.~F. Roth.
\newblock On a problem of {H}eilbronn.
\newblock {\em J. London Math. Soc.}, 26:198--204, 1951.

\bibitem[Sta96]{rS96}
Richard~P. Stanley.
\newblock {\em Combinatorics and commutative algebra}, volume~41 of {\em
  Progress in Mathematics}.
\newblock Birkh\"auser Boston Inc., Boston, MA, second edition, 1996.

\end{thebibliography}

\end{document}